\documentclass[preprint,12pt,authoryear]{elsarticle}

\journal{Journal of Algebra}

\usepackage{amssymb}
\setcitestyle{square,sort,comma,numbers}

\usepackage{graphicx}

\usepackage{lineno,hyperref}
\usepackage{amssymb}
\usepackage[top=1in, bottom=1in, left=1in, right=1in]{geometry}
\usepackage{amssymb,amsfonts}
\usepackage{bbm}
\usepackage[all,arc]{xy}
\usepackage{enumerate}
\usepackage{mathptmx}
\usepackage{mathrsfs}
\usepackage[T1]{fontenc}
\usepackage[applemac]{inputenc}
\usepackage[english]{babel}
\usepackage{times}
\usepackage{multirow}
\usepackage{color}
 \usepackage[english]{babel}
\usepackage{graphicx}
\usepackage{float}
\usepackage{xcolor}
 \usepackage{amsmath}
\usepackage{graphicx}
\usepackage{multicol}
\usepackage{amsmath}
\usepackage{subeqnarray}
\usepackage{url}
\usepackage[bottom]{footmisc}

 \usepackage{amsthm}
\usepackage{relsize}
\usepackage{bm}
\usepackage{yfonts}
\usepackage[mathcal]{euscript}

\usepackage{etex}

\usepackage{braket}
\usepackage{tikz-cd}

\usepackage{verbatim}

\usepackage{graphicx}

\newtheorem{thm}{Theorem}[section]
\newtheorem{cor}[thm]{Corollary}
\newtheorem{prop}[thm]{Proposition}
\newtheorem{lem}[thm]{Lemma}

\theoremstyle{definition}

\newtheorem{examp}{Example}

\newtheorem{rem}{Remark}

\modulolinenumbers[5]
\usepackage{psfrag}
\usepackage[inline]{enumitem}   
\makeatletter

 \DeclareMathOperator{\Tr}{Tr}
\DeclareMathOperator{\Id}{Id}
\DeclareMathOperator{\Aut}{Aut}

\begin{document}
\begin{frontmatter}

\title{Linear  representation of groups  associated to graphs}

\author[A]{Alain Bretto\corref{cor1}}
\ead{alain.bretto@unicaen.fr}
\cortext[cor1]{Corresponding author}
\author[B]{Alain Faisant}
\ead{faisant@univ-st-etienne.fr}
\author[A,C]{Neda Moradi}
\ead{neda.moradi@unicaen.fr}
\author[A,D]{Mehrdad Nasernejad}
\ead{m_nasernejad@yahoo.com}
\address[A]{Normandie Univ-Caen, GREYC CNRS UMR-6072, Campus II, Bd Marechal Juin BP 5186, 14032 Caen cedex 5, France}
\address[B]{Univ Lyon, UJM-Saint-\'Etienne, CNRS, ICJ UMR 5208, 42023 Saint-\'Etienne, France}
\address[C]{University of Kashan, Faculty of Mathematical Sciences, Kashan, Iran}
\address[D]{Univ. Artois, UR 2462, Laboratoire de Math\'ematique de Lens (LML), F-62300 Lens, France}

\begin{abstract}
In this article, we introduce a group  linear representation associated to a graph. 
We study this representation as well as the canonical decomposition of the associated module. A character study is conducted, demonstrating the relevance of this approach to graph theory.  We also develop the links that exist between algebra and graphs, particularly in the language of representations.

\end{abstract}


 \end{frontmatter}


\section{Introduction}

Graph theory can be approached through different branches of mathematics. One of the first approaches is naturally combinatorics, but there are also other aspects of mathematics that allow for relevant analysis of graphs, \cite{bre3}.  The topological aspect is well developed and provides high-quality results and applications in various fields, such as computer graphics, \cite{gross,archdeacon}. Although more recent, calculus or analysis on graph has developed rapidly thanks to applications in image analysis and data analysis, \cite{schott}. Algebra in graph theory, on the other hand, is older and gave rise to algebraic graph theory, \cite{brouwer,godsil,beineke}.
The algebraic theory of graphs mainly comprises three fundamental axes: the use of linear algebra, the use of group theory, and the study of graph invariants. What is curious is that virtually no other areas of algebra, which is an extremely rich branch of mathematics, are addressed in graph theory.\\
Group theory, primarily through the group of automorphisms and its subgroups, allows us, mainly, to study graph symmetries in the general sense of the term. Furthermore, results such as \textsc{Fruch}'s theorem highlight the fundamental relationships between groups and graphs, \cite{bre3}.\\
The spectral theory of graphs is undoubtedly the one that gives the most significant results, thanks in particular to the study of the eigenvalues of the adjacency matrix, but also to the eigenvalues of  the laplacian.\\
The invariants, for their part, are in most cases approached from the perspective of specific polynomials such as:
\begin{itemize}
\item[-] the chromatic polynomial;
\item[-] the polynomial of matchings;
\item[-] the polynomial of stable;
\item[-] \textsc{Tutte}'s polynomial;
\item[-] and many others.
\end{itemize}
In this work, we present a new algebraic framework that establishes a deep link between graph theory and group representation theory. Our approach transposes the combinatorial problem of graph isomorphism into the algebraic domain of module equivalence, \cite{ABF}, thus exploiting the power of character theory and group representations.\\
The central construction associates to each graph $\Gamma$ a natural module structure on a group algebra $R = \overline{\mathbb{F}}_p[\mathbb{E}]$, where $\mathbb{E}$ denotes the edge group of the graph and $\overline{\mathbb{F}}_p$ is an algebraic closure of the field $\mathbb{F}_p$, $p$ being an odd prime number. This module results from the canonical action of the edge group on the vertex space, extending the incidence structure of the graph to a linear group representation that fully characterizes the graph under study.

Our main theoretical contribution establishes a fundamental equivalence: two graphs are isomorphic if and only if their associated modules are isomorphic, which in turn holds if and only if their characters coincide. This triple equivalence provides  a conceptual bridge between graph theory and linear representation theory of group.

As a concrete application of this framework, we introduce a new graph invariant - the signature- which is related to the character of the representation. This allows us to study the topology of the analyzed graph. Furthermore, thanks to the character, we can highlight properties of the graph such as Eulerianity, matching, cycles, and so on.


\section{Basic defintions and results}

\subsection{Graphs}
 We define a graph $\Gamma=(V; E; \epsilon)$ as follows:
\begin{itemize}
\item $V$ is the set of vertices and $E$ is the set of edges.
\item $\epsilon$ is a map from $E$ to $P_{2}(V)$, where $P_{2}(V)$ is the set of subsets of $V$ having $1$ or $2$ elements.
\end{itemize}

In this paper graphs are finite, i.e., sets $V$ and $E$ have
finite cardinalities. For each edge $a$, we denote $\epsilon (a) =
[x; y]$ if $\epsilon (a) =\{x, y\}$ with $x\neq y$ or $\epsilon
(a) =\{x\}$ if $x = y$, in this case  $a$ is a  {\bf loop}. The
elements $x, y$ are called extremities of $a$, and $a$ is incident
to $x$ and $y$.
The set $\{a\in E, \epsilon(a) = [x; y]\}$ is called {\bf multi-edge} or {\bf $p$-edge}, where $p$ is the cardinality of the set, if $p=1$ and there is no loop, the graph is {\bf simple}.
We define the {\bf degree} of $x$ by $d(x) = card(\{a\in E, x\in \epsilon (a) \})$.\\
Given a graph $\Gamma=(V; E; \epsilon)$, a {\bf chain}, noted $L_{k}$ is a
non-empty graph $P =(V, E, \epsilon)$ with $V = \{x_0, x_1, \ldots
, x_k\}$ and $E = \{a_1,a_2, \ldots ,a_{k-1} a_k\}$, where
$x_{i},x_{i+1}$, ($i \mod k$) are extremities of $a_{i}$. The
elements of $E$ must be distinct. The cardinality of $E$, $k$  is the
{\bf length} of this chain.
A {\bf  cycle}, noted $C_{k}$ is a chain such that $x_0= x_k$.  A graph is {\bf Eulerian}  if there exists a cycle that passes through all the vertices exactly once.\\
A graph is {\bf connected} if, for all $x, y\in V$, there exists a chain from $x$ to $y$.\\
$\Gamma^{'} = (V'; E'; \epsilon^{'})$ is a {\bf subgraph} of $\Gamma$ if it is a graph satisfying
$V' \subseteq V$, $E' \subseteq E$ and $\epsilon^{'}$ is the restriction from $\epsilon$ to $E'$.
If $V' = V$ then $\Gamma^{'}$ is a {\bf spanning subgraph}.\\
An {\bf induced subgraph} generated by $A$, $\Gamma(A)=(A; U; \epsilon)$, with $A \subseteq V$ and $U \subseteq E$
is a subgraph such as $U = \{ a\in E, \epsilon (a) = [x; y], x, y \in A \}$.\\
An induced subgraph such that any pair of vertices are adjacent is
called a {\bf clique}. Let $\Gamma=(V; E; \epsilon)$ be a graph,
a {\bf component} of $\Gamma$ is a maximal connected induced
subgraph.\\
Let $\Gamma_{1} =(V_{1}; E_{1}; \epsilon_{1})$ and $\Gamma_{2}
=(V_{2}; E_{2}; \epsilon_{2})$ be two graphs, a {\bf morphism}
from $\Gamma_{1} =(V_{1}; E_{1}; \epsilon_{1})$ to $\Gamma_{2}
=(V_{2}; E_{2}; \epsilon_{2})$ is a couple $(f, f^{\#})$ where
$f:\; V_{1}\;\longrightarrow\; V_{2}$ is a map and $f^{\#}:\;
E_{1}\;\longrightarrow\; E_{2}$ is a map such that:
\[ \text{if}\; \epsilon_{1} (a) = [x; y]\; \text{then}\; \epsilon_{2} (f^{\#}(a)) = [f(x); f(y)]. \]
So $(id_{V}, id_{E})$ is a morphism from $G =(V; E; \epsilon)$ to $G$.\\
The {\bf composition} of two morphisms $(f, f^{\#})$ and $(g,
g^{\#})$ is defined by: $(f, f^{\#})\circ (g, g^{\#}):= (f \circ
g, f^{\#}\circ g^{\#})$. $(f, f^{\#})$ is an {\bf isomorphism} if
there exists a morphism $(g, g^{\#})$ from $\Gamma_{2} =(V_{2};
E_{2}; \epsilon_{2})$ to $\Gamma_{1} =(V_{1}; E_{1};
\epsilon_{1})$ such that $(g, g^{\#})\circ (f, f^{\#})=
(id_{V_{1}}, id^{\#}_{E_{1}})$ and
$(f, f^{\#})\circ (g, g^{\#}) = (id_{V_{2}}, id^{\#}_{E_{2}})$.
In this case we will denote $(g, g^{\#}) = (f, f^{\#})^{-1}$. So $(f, f^{\#})$ is an
isomorphism if and only if $f$ is a bijection and $f^{\#}$ is a bijection.
If there exists an isomorphism between $\Gamma_{1}$ and $\Gamma_{2}$ we will denote $\Gamma_{1}\simeq \Gamma_{2}$ and
we will say that $\Gamma_{1}$ is {\bf isomorphic} to $\Gamma_{2}$.\\
A {\bf matching} in a graph is a set of edges without common vertices.\\
A {\bf clique} in a graph $\Gamma$  is a subset of vertices  such that every two distinct vertices are adjacent.


\subsection{Representation}
Let  $G$  be a finite group, 
and $k$ a fixed field; a  finite {\bf linear $k$-representation} of $G$ is  a group-morphism 
\[
\rho \colon G \to \mathrm{GL}(V)
\]
where $V$ is a finite dimensional $k$-vector space, and  $GL(V)=Aut_k(V)$ the  linear group of $V$; the dimension of $V$ is the {\bf degree} of $\rho$; one also say that $V$ is a {\bf representation} of $G$.\\

Let $\rho \colon G \to \mathrm{GL}(V),\; \rho' \colon G \to \mathrm{GL}(W)$ be two representations of $G$;
 $\rho$ and $\rho'$ are {\bf  equivalent} if there exists a $k$-linear isomorphism $\alpha: V\longrightarrow W$ verifying:
\[
\alpha\circ\rho= \rho'\circ\alpha
\]
that is to say that
\[
\alpha(\rho(g))(x)= \rho'(g)(\alpha(x))\; \textrm{for every}\;  g\in G\; \textrm{and every}\; x\in V.
\]
The {\bf character} of $\rho$ is the (central) function $\chi_{\rho}$ defined by

$\begin{array}{ccccl}
&\chi_{\rho}(g) : & G & \longrightarrow  & k\\
 & & g & \longmapsto  &\chi_{\rho}(g)= \Tr(\rho(g))\\
\end{array}$\\

\noindent Clearly two equivalent representations have the same character. As is well known the converse is true if  $k= \mathbb{C}$, that is, when $car(k)=0$; when $car(k)=p $ the situation is more delicate.\\\\

Let $k[G]$ be the  {\bf group algebra} over the field $k$:

\[
 k[G]= \left\{\sum_{g\in G} a_{g}g: a_{g}\in k\right\}
\]
$G$ is considered as a basis of $k[G]$.\\\\
If $V$ is a representation of $G$, the action of $G$ extends by $k$-linearity  to an action of $k[G]$: $V$ becomes a $k[G]$-left module.\\
Conversely, every $k[G]$-left module (of finite dimension over $k$)  defines a linear representation of $G$: these two languages are useful. 
In particular $\rho$ is equivalent to $\rho'$ iff the $k[G]$-modules associated are isomorphic.\\

A $k[G]$-module $M$ is {\bf simple} if the only submodules of $M$ are $\{0\}$ and $M$; the associated representation is {\bf irreducible}.\\
When $k$ is algebraically closed and $G$ is abelian then every irreducible representation has dimension $1$. \\
For more information on group representations see \cite{weintraub,serre1,serre2,gorenstein,alperin}.


\subsection{Dual}
Let $G$ be an abelian group of exponent $u$, that is $g^u=1$ for every $g \in G$, and fix 
 $C_u$ is a cyclic group with $u$ elements (cf. \cite{lang} p. 50):  $X(G)=Hom(G,C_u)$ is the "dual group" of $G$.\\

In our text we will use the group $(\mathbb{Z}/2\mathbb{Z})^m$ denoted by $G=\mathbb{E}_m$: here $u=2$ and choose $C_u=\{\pm 1 \}$.  The law  is noted $\oplus$: $\mathbb{E}$ is also a $\mathbb{F}_2$-vector-space of dimension $m$ with canonical basis $E=\{e_1, \ldots , e_m\}$ so that the elements have a unique expression 
${\displaystyle   a= \oplus_{1 \le i \le m} \alpha_i e_i}$ with $\alpha_i= 0, 1$.\\
In these conditions we have an explicit isomorphism:

\begin{prop}\mbox{}\\
\begin{itemize}
\item[i)]  ${\displaystyle \psi_E : \mathbb{E} \longrightarrow X(\mathbb{E})}$ defined, for $a =\sum_{e \in E}a_e e$ by:

\[
\psi_E(a)=\chi_a : \;\; \chi_a(e)= (-1)^{a_e}\; \textrm{where}\; a_{e}=0\; \textrm{if}\; e\not\in \sigma(a); \;  a_{e}=1\; \textrm{otherwise}, 
\]

is an isomorphism; 
\item[ii)] $\psi^{-1}_{E} : X(\mathbb{E}) \longrightarrow \mathbb{E}$ is the following : 
$$\psi^{-1}_{E}( \chi_{a})= a= \oplus_{e \in E} a_e e; \;\; a_{e}=0\; \textrm{or}\; 1, \; \textrm{where} \;a_e\;\textrm{is defined by} \; \chi(e)=(-1)^{a_e}.$$
\end{itemize}
\end{prop}

\begin{proof}\mbox{}\\
\begin{itemize}
\item[i)] It is easy, by applying the definition, to prove  that $\psi_E$ is a morphism; this morphism is injective since $\chi_a=1$ implies $a_e=0$ for all $e$, therefore $a=0$.\\ Since the groups are finite, and moreover it is well known that: $\vert G\vert= \vert X(G)\vert$, we conclude that: $\psi_E$ is an isomorphism;
\item[ii)] 

Let $\varphi(\chi)= \oplus x_e e,$ where $x_e=0$ or  $x_e=1$, and defined by $\chi(e)=(-1)^{x_e}$. \\
Then $ \varphi \circ \psi_E(a)=\varphi(\chi_a)=\oplus x_e e,$ where $\chi_a(e)=(-1)^{x_e}$ ; but $\chi_a(e)=(-1)^{a_e}$ so that $x_e=a_e$ for all $e$ and $\oplus x_e e =a$: hence
$\varphi \circ \psi_E=Id_{\mathbb{E}}$ and $\varphi= \psi_E^{-1}.$ \qedhere
\end{itemize}
\end{proof}

In the following $\mathbb{F}_q$,  $q$ prime, design the  \textsc{Galois} field with $q$ elements.


\section{$\overline{\mathbb{F}}_p$-representation of  a graph }
We fix an odd prime $p$ and we choose $k= \overline{F}_p$ an algebraic closure of $\mathbb{F}_p$.\\

Let $\Gamma=(V; E, \varepsilon)$ be a simple graph, with $n =\vert V\vert$,  $m =\vert E\vert$,   $m \ge 1$. 
\begin{itemize}
\item[-] To the set of edges $E$, we associate:

$$\mathbb{E}(\Gamma)=\mathbb{E}= \mathbb{F}_2[E].$$

An element of $\mathbb{E}$ can be written
 $$a=\oplus_{e \in E}a_e e, \; a_e \in \mathbb{F}_2$$

the element  $\sigma(a)=\{ e: a_e \neq 0 \}$ is the {\bf support}  of $a$.

 \item[-] To the set of vertices $V$ we associate   
$$\mathbb{V}(\Gamma)=  \mathbb{V}= \overline{\mathbb{F}}_p[V]$$
An element of $\mathbb{V}$ is written 
$ \xi=\sum_{x \in V}\lambda_x x, \;\;  \lambda_x \in \overline{\mathbb{F}}_p $: $\mathbb{V}$ is the $n=|V|$-dimensional
$\overline{\mathbb{F}}_p$-vector space  with $V$ as  basis.\\\\
For the dual group $X(\mathbb{E})=Hom( \mathbb{E}, C_u)$ we choose $C_u=\{\pm1\}$ the  subgroup of  $\mathbb{F}_p^{\times}$ (recall that $p$ is odd).
 \item[-] Let's define the following function:
\begin{equation}\label{gamma}
\begin{array}{ccccc}
\gamma & : & V  & \longrightarrow & \mathbb{E}\\
              &   &  x & \longmapsto      & \underset{\varepsilon(e) \ni x}{\oplus}e 
\end{array}
\end{equation}
hence,  $\gamma(x)=0$ if $x$ is isolated.
 \item[-]   To $\gamma(x)$ we associate the homomorphism $\psi_E(\gamma(x))= \chi_{\gamma(x)} \in X(\mathbb{E})$:
 $$\chi_{\gamma(x)}: \mathbb{E} \longrightarrow \{\pm 1\} \subset \mathbb{F}_p^{\times}$$
which is defined on the basis $E$ by $\chi_{\gamma(x)} (e)= (-1)^{\vert\sigma(\gamma(x))\cap e\vert}$, that is:

\begin{equation}\label{carac}
\chi_{\gamma(x)}(e) =
\begin{cases}
-1 & \text{if }x \in \varepsilon(e)\\
\;\;\; 1& \text{otherwise}.
\end{cases}
\end{equation}

For $u \in \mathbb{E}$  with support $\sigma(u)$, we denote $\Gamma(u)$  the subgraph induced by the edges $e \in \sigma(u)$, i.e. $\Gamma(\sigma(u))$.\\

Let 
\[
a=\oplus_{e \in E}a_e e, \; a_e \in \mathbb{F}_2\; \textrm{and } u=\oplus_{e \in E}b_e e, \; b_e \in \mathbb{F}_2; 
\]
we have: 
$$
 \chi_{a}(u)=\chi_{a}(\underset{e \in \sigma(u)}{\oplus}b_{e}e)\\
= \prod_{e \in \sigma(u)}\chi_{a}(e)=\prod_{e \in \sigma(u)\cap \sigma(a)} (-1)^{a_{e}}
= (-1)^{\vert \sigma(u)\cap \sigma(a)\vert}.
$$

It implies that 
\begin{equation}\label{ki}
 \chi_{\gamma(x)}(u)= (-1)^{\vert \sigma(a)\cap \gamma(x)\vert}= (-1)^{deg_{\Gamma(u)}(x)}.
\end{equation}

\item[-]  Now, we define: 
\begin{equation}\label{rho}
 \rho_{\Gamma}: \mathbb{E} \longrightarrow Aut(\mathbb{V})=GL(\mathbb{V}) 
 \end{equation}
by $\rho_{\Gamma}(a)(x)=\chi_{\gamma(x)}(a)x,\;   \textrm{where}\;   a\in \mathbb{E}\; \textrm{and}\;   x \in V.$ \\
Hence,  $\rho_{\Gamma}(a)$ extends to a  $\overline{\mathbb{F}}_p$-linear endomorphism of $\mathbb{V}$, and actually 
$\rho_{\Gamma}(a) \in Aut(\mathbb{V})$, since $\rho_{\Gamma}(a)^2=Id_{\mathbb{V}}$. \\\\

The  representation  {\bf $\rho_{\Gamma}$ is  the $\overline{\mathbb{F}}_p$-representation of $\mathbb{E}$ associated to $\Gamma$}. Its  degree is $n=|V|= dim_{\overline{\mathbb{F}}_p} (\mathbb{V})$.
\end{itemize}


\section{$R$-module associated to a simple graph}
We define now: 
$$ R= \overline{\mathbb{F}}_p[\mathbb{E}], $$
the group algebra of $\mathbb{E}$ (recall that we have chosen  $k=\overline{\mathbb{F}}_p$).\\
So $\mathbb{V}=\mathbb{V}(\Gamma)$ becomes an $R$-module, where the  external law is given by
\begin{equation}
r.w=\rho_{\Gamma}(r)(w).
\end{equation}

Here ${\displaystyle r\in R= \overline{\mathbb{F}}_p[\mathbb{E}]}$ is written: 
$${\displaystyle r=\oplus_{a\in \mathbb{E}}r_a.a , \; \; r_a \in \overline{\mathbb{F}}_p};$$
if ${\displaystyle s = \oplus_{a\in \mathbb{E}}s_a.a}$, then 
$${\displaystyle r\oplus s=\oplus_{c \in \mathbb{E}}(\oplus_{a\oplus b=c} r_a.s_b).c}$$
and for ${\displaystyle \lambda \in \overline{\mathbb{F}}_p: \lambda r =\oplus \lambda r_a.a}$: vector-space structure of $R$. \\
The ring $R$ is of characteristic $p$. \\


\section{Canonical decomposition of the $R$-module $\mathbb{V}(\Gamma)$}\label{decomposition}

We start with $\mathbb{V}(\Gamma)= \bigoplus_{x \in V} \overline{\mathbb{F}}_p x$. The set $\mathbb{V}(\Gamma)$ is an  $R$-module: $a.x=\rho_{\Gamma_m}(a)(x)=\chi_{\gamma(x)}(a)x$, $x \in V$; moreover,  $\overline{\mathbb{F}}_p x$ is a subrepresentation of $\rho_{\Gamma}$: the action of $e \in E$ is given by $e.(\lambda x)=\lambda (e.x)=\lambda (\pm x) \in  \overline{\mathbb{F}}_px$; being of dimension $1$ we see that $ \overline{\mathbb{F}}_px$ is irreducible, that is a simple $R$-module. We want to specify simple $R$-modules which are ideals of the ring $R$; for this we define for $u\in \mathbb{E}$:
$$
\xi_u= \sum_{a \in \mathbb{E}}\chi_u(a)a \in  \overline{\mathbb{F}}_p[\mathbb{E}]=R.
$$


\begin{lem}\label{idealPrincipal} \mbox{}
\begin{itemize}
\item[i)] For $b \in \mathbb{E}$: $\xi_u \oplus b = \chi_u(b) \xi_u$.
\item[ii)] $ \overline{\mathbb{F}}_p \xi_u= \xi_u \oplus R= (\xi_u)$ is a principal ideal of $R$ and a simple $R$-module.\\
Its character is $\chi_u$.
\end{itemize}
\end{lem}

\begin{proof}\mbox{}
\begin{itemize}
\item[i)]  for all $b \in \mathbb{E}:\xi_u \oplus b=\underset{a \in \mathbb{E}}{\oplus}\chi_u(a)a\oplus b$; we proceed with the following variable change: $c=a\oplus b$;
since $b \in \mathbb{E}$ we obtain: $c\oplus b= a$; thus, 
\[
\xi_u\oplus b=\oplus_{c} \chi_u(c\oplus b)c=\chi_u(b) \oplus_c\chi_u(c) c=\chi_u(b)\xi_u.
\]
\item[ii)]  As consequence of i) we have $\overline{\mathbb{F}}_p\xi_u= \xi_u \oplus R= (\xi_u)$,\\
and the representation associated to $(\xi_u)$ is given by $a \longmapsto \rho(a),$ 
where $\rho(a)(\xi_u)= \chi_u(a)\xi_u$; so, his character is $\chi_u$.\qedhere
\end{itemize}
 \end{proof}


\begin{prop}
Let $\Gamma =(V, E, \varepsilon)$ be a simple graph. Then for $x \in V:$
\[
\overline{\mathbb{F}}_px \simeq \xi_{\gamma(x)} \oplus R= \left( \xi_{\gamma(x)}\right)
\]
as simple $R$-modules, with character $\chi_{\gamma(x)}$.
\end{prop}

\begin{proof} Define the following function:
$$
\begin{array}{cccl}
f : &  \overline{\mathbb{F}}_px& \longrightarrow &  \xi_{\gamma(x)}\oplus R=  \overline{\mathbb{F}}_p\xi_{\gamma(x)} \\
  & \lambda x & \longmapsto& f(\lambda x)= \lambda \xi_{\gamma(x)}  \\
\end{array}
$$

$f$ is a $\overline{\mathbb{F}}_p$-linear bijection; and a morphism of $R$-modules: for $r=b \in \mathbb{E}$:
\begin{itemize}
\item[-] $f(b. \lambda x)=f(\lambda b .x)= f(\lambda \chi_{\gamma(x)}(b) x)=\chi_{\gamma(x)}(b)\lambda \xi_{\gamma(x)}$;
\item[-] $b.f(\lambda x)=b. \lambda \xi_{\gamma(x)}=\lambda \xi_{\gamma(x)}\oplus b =\lambda\chi_{\gamma(x)}(b)\xi_{\gamma(x)}$ by  Lemma \ref{idealPrincipal} ; so $f$ is an isomorphism of modules.\qedhere
\end{itemize}
\end{proof}

Consequently, we get 
\begin{equation}
\mathbb{V}(\Gamma)\simeq \bigoplus_{x \in V} \left(\xi_{\gamma(x)}\oplus R\right).
\end{equation}

We are looking for the isomorphic components of this decomposition; for this we define 
\begin{eqnarray*}
I_V                    &=& \{x \in V : x \textrm{\;isolated vertex} \} ,\; \vert I_V \vert =m_{is},\\
I_{E}&=& \{e \in E : e \textrm{\;isolated edge} \},\\
V_{\gamma}    &= & V \setminus \left\{I_V \cup \varepsilon (I_{E})\right\}.
\end{eqnarray*}
From this, we get: $\gamma(x) = 0 $ on $I_V$, and if $\varepsilon(e)=\{x, y\}$ is an isolated edge, then $\gamma(x)=\gamma(y)$.


\begin{thm}\label{canondecomposition}\textrm{} \\
Let $\Gamma=(V;E, \varepsilon)$ be a simple  graph; then the  splitting of $\mathbb{V}(\Gamma)$ as an $R$-module 
\begin{equation}
\mathbb{V}(\Gamma) \simeq m_{is} \left(\xi_{0}\oplus R\right)\bigoplus_{e\in I_{E}} 2 \left(\xi_{e}\oplus R\right) \bigoplus_{x \in V_{\gamma}} \left(\xi_{\gamma(x)}\oplus R\right),
\end{equation}
is a decomposition into simple non isomorphic $R$-modules, witch are ideals of $R$.
\end{thm}
We call this splitting the {\bf canonical decomposition} of $\mathbb{V}$.

\begin{proof}\mbox{}
\begin{itemize}
\item[$\bullet$] For $x\in I_V$:  we have $\gamma(x) = 0$ and  $\mathbb{F}_p x$ is the unit representation, (or   the trivial representation),  $(\xi_{0}\oplus R)$;
\item[$\bullet$] for  $e \in I_{E}, \varepsilon(e)=\{x, y \}$, we have $ \gamma(x)=\gamma(y)=e$. In this case, 
$$\mathbb{F}_p x\simeq \mathbb{F}_p y \simeq \xi_{\gamma(x)}\oplus R=\xi_{\gamma(y)}\oplus R= \xi_e \oplus R;$$
\item[$\bullet$] we now have to show that $\gamma$ is injective on $V_{\gamma}$ ; let $x\neq y 
\in V_{\gamma}$; note  that there  exists at most one $e \in E$ such that $\varepsilon(e)=\{x, y\}$; 

\begin{itemize}
\item[-] if $x, y$ are not adjacent, then $\gamma(x) \neq \gamma(y)$ and  $ \xi_{\gamma(x)}\oplus R \not\simeq \xi_{\gamma(y)}\oplus R$;
\item[-] if $\{x,y \}=\varepsilon(e)$ by hypothesis $deg(x) \ge 2$ or $deg(y) \ge 2$; if for instance  $deg(x) \ge 2$, there exists $e'\neq e$ incident to $x$, so $e' \in \gamma(x) \setminus \gamma(y)$, and  $ \xi_{\gamma(x)}\oplus R \not\simeq \xi_{\gamma(y)}\oplus R$.
\end{itemize}
\end{itemize}
That concludes the proof.
\end{proof}
It is easy to verify that:
\begin{equation}
n=  m_{is}+ 2m_{I_{E}}+ m_{\gamma},
\end{equation}
where $m_{is}$ is the isolated vertex number; $m_{I_{E}}$ is the isolated edge number and $m_{\gamma}$ is the coefficient of $ \left(\xi_{\gamma(x)}\oplus R\right)$, that is, either $0$ or $1$.


\section{Equivalence of representations}
In order to compare the representations $\rho_{\Gamma}$ of different simple graphs $\Gamma$, we must consider graphs having the same labeling of edges, this will allow us to have the same group:
\[
 \mathbb{E}= \overline{\mathbb{F}}_p[E].
 \]


\subsection{m-graphs and m-isomorphisms} 

Let $\mathcal{E}_m=\{1, 2, \ldots, m \}$; an {\bf $m$-graph} is a graph $\Gamma_{m}=(V, \mathcal{E}_m, \alpha)$: the set of edges is labeled by $\mathcal{E}_m$.\\

\noindent Every  graph $\Gamma=(V;E, \varepsilon)$ with $m$ edges has a "copy''  which is an $m$-graph:
for that define $\varphi =\Id$, and choose   a bijection $\varphi^{\#}: E \longrightarrow \mathcal{E}_m$; define now 
$\Gamma_{m}=(V; \mathcal{E}_m, \alpha) $ by the rule $\alpha(i)= \alpha(\varphi^{\#}(e)):= \varepsilon(e)$; then 
$(\varphi, \varphi^{\#}) : \Gamma \longrightarrow \Gamma_m$ is an isomorphism.\\\\
Let $(f,f^{\#})  : \Gamma_m=(V;\mathcal{E}_m ,\alpha) \longrightarrow \Gamma'_m=(V';\mathcal{E}_m,\alpha')$ be an isomorphism 
($\alpha' \circ f^{\#}=f \circ \alpha$) ; it is called an  {\bf $m$-isomorphism} if  $f^{\#}=\Id: \alpha' =f \circ \alpha$.

\begin{rem} An isomorphism is not necessarily an $m$-isomorphism : for instance,  let 
$\Gamma_m =\Gamma'_m =(V;\mathcal{E}_m, \alpha)$ the graph with $V=\{x, y, z \}, m=2$ and $\alpha(1)=\{x, y\}, 
\alpha(2)=\{y, z \}$ (the connected graph $L_2$ with $2$ edges) then 
$f(x)=z, f(y)=y, f(z)=x, f^{\#}(1)=2, f^{\#}(2)=1$ is an automorphism: $\alpha \circ f^{\#}(1)= \alpha(2)=\{y,z\}=f(\{x,y\})=f\circ \alpha(1),
\alpha \circ f^{\#}(2)= \alpha(1)=\{x,y\}=f(\{y,z\})=f\circ \alpha(2)$, and  not an $m$-automorphism: $f^{\#} \neq Id$.\\
However, for two isomorphic graphs $\Gamma=(V;E, \varepsilon)$ and $\Gamma'=(V';E', \varepsilon')$, we can always find two copies of these graphs that are $m$-isomorphic, this is the subject of the next Proposition
\end{rem}


\begin{prop} \label{m-grahsIso} Let $\Gamma=(V;E,\varepsilon)$ and  $\Gamma'=(V';E', \varepsilon')$ be two simple graphs with $\vert V\vert =\vert V'\vert =n$ and  $\vert E\vert =\vert E'\vert =m$; the two following properties are equivalent:
\begin{enumerate}
\item[i)] the graphs $\Gamma=(V;E,\varepsilon)$ and  $\Gamma'=(V';E', \varepsilon')$ are isomorphic;
\item[ii)]  there exist copies 
 $\Gamma_m=(V; \mathcal{E}_m, \alpha), \Gamma'_m=(V';\mathcal{E}_m,\alpha')$ which are $m$-isomorphic.
\end{enumerate}

\end{prop}

\begin{proof}
Let $(f,f^{\#} ): \Gamma \longrightarrow \Gamma'$ be  an isomorphism.
\begin{itemize}
\item[-] We choose a bijection $\varphi^{\#}: E \longrightarrow \mathcal{E}_m$ and we define  $\alpha(i)=\alpha(\varphi^{\#}(e)) := \varepsilon(e)$; this gives $\Gamma_m=(V;\mathcal{E}_m,\alpha)$;
\item[-] now $\psi^{\#}:= \varphi^{\#} \circ (f^{\#})^{ -1}: E' \longrightarrow \mathcal{E}_m$ is a bijection; furthermore we define
$\alpha'(i)=\alpha(\psi^{\#}(e')) := \varepsilon'(e')$,  this gives $\Gamma'_m=(V';\mathcal{E}_m,\alpha')$;
\item[-]  then $(f,\Id): \Gamma_m \longrightarrow \Gamma'_m$ is an  $m$-isomorphism, i.e. $\alpha'=f \circ \alpha$:
for $1 \le i \le m: i=\varphi^{\#}(e)$ we have $\alpha(i)=\varepsilon(e), f\circ \alpha(i)=f \circ \varepsilon(e)=
\varepsilon' \circ f^{\#}(e)$; but $f^{\# -1}=\psi^{\# -1} \circ \varphi^{\#}$, so that $f \circ \alpha(i)=
\varepsilon'(\psi^{\#-1}(\varphi^{\#}(e))=\varepsilon'(\psi^{\#-1}(i)$. denote $e':= \psi^{\#-1}(i)$:
$f \circ \alpha(i)= \varepsilon'(e')=\alpha'(\psi^{\#}(e'))=\alpha'(i)$.
\end{itemize}
Consequently,  the following diagram is commutative:

\[ 
\begin{tikzcd}[sep=2cm]
 \Gamma_m= (V;\mathcal{E}_m,\alpha) \arrow{r}{(f,Id)} \arrow[swap]{d}{(Id,\varphi^{\#-1})} &\Gamma'_m= (V'; \mathcal{E}_m,\alpha')  \arrow{d}{(Id,\psi^{\#})} \\%
\Gamma  =(V;E,\varepsilon)  \arrow{r}{(f,f^{\#})}& \Gamma'=(V';E',\varepsilon')
\end{tikzcd}
\]\\
The converse is obvious.
\end{proof}


\subsection{Representations  of $\mathbb{E}_m$ }
We define: 
$$\mathbb{E}_m:=\mathbb{F}_2[\mathcal{E}_m] \;\;\;\;   R_m:= \overline{\mathbb{F}}_p[\mathbb{E}_m]$$

In this context  two  $\overline{\mathbb{F}}_p$-representations  $\rho: \mathbb{E}_m \longrightarrow Aut( W)$ and $\rho' : \mathbb{E}_m \longrightarrow Aut(W')$ of the group $\mathbb{E}_m$ are equivalent, denoted by $\rho \simeq \rho'$,  if there exists $\varphi :W \longrightarrow W'$ a $ \overline{\mathbb{F}}_p$-isomorphism such that for all $a \in {E}_m$
$$\varphi\circ \rho (a)= \rho'(a)\circ \varphi,$$
that is, $\varphi$ is an isomorphism of $R_m$-modules.


\subsection{Excellent field}
The field $\overline{\mathbb{F}}_p$ is an algebraic closure of $\mathbb{F}_p$, moreover p is coprime with the
 order of $\mathbb{E}_m$, therefore $\overline{\mathbb{F}}_p$ is an excellent field for $\mathbb{E}_m$ (see \cite{weintraub} p. 31, p.56); moreover, since $\mathbb{E}_m$ is an abelian group, \textsc{Frobenius}'s Theorem (see \cite{weintraub}  p. 57) states that a complete set of non-isomorphic $R$-simple modules has cardinality $|\mathbb{E}_m|$.\\\\

\noindent  The field $\overline{\mathbb{F}}_p$  being excellent, applying the  Corollary 5.22 (2) p 74  in \cite{weintraub} we obtain:


\begin{thm}\label{frobenius} Let $(W_{\alpha})_{\alpha \in A}$ be a set of non-isomorphic simple $R$-modules; 

\[
\textrm{if}\quad  W=\underset{\alpha \in A}{\oplus}m_{\alpha} W_{\alpha}\;  \textrm{ and }\;  W'=\underset{\alpha \in A}{\oplus}m'_{\alpha}W_{\alpha} 
\]
 are splitting decompositions with $\chi_W=\chi_{W'}$
then
\[
\forall \; \alpha \in A \;\; m_{\alpha} \equiv m'_{\alpha}  \mod p.  
\]
\end{thm}


\subsection{Characters}\mbox{}\\
As consequence of Theorem  \ref{frobenius} we see that irreducible representations with the same character are equivalent; but we know $|\mathbb{E}_m|$  
irreducible representations:
 \[
 (\xi_a \oplus R_m)_{a \in \mathbb{E}_m}
 \]
 so this set is a {\bf complete system} of simple $R_m$-modules made up of ideals of $R_m$.\\

Our canonical decomposition uses precisely this system.


\begin{thm} \label{graphEquivModules}
Let $R_m=\overline{\mathbb{F}}_p[\mathbb{E}_m]$, with  $\mathbb{E}_m=\mathbb{F}_2[\mathcal{E}_m], \; m \ge 2$.\\
 Let  $\Gamma_m=(V;\mathcal{E}_m, \alpha)$ and $\Gamma'_m=(V'; \mathcal{E}_m, \alpha')$ be two simple $m$-graphs. The following properties are equivalent:
\begin{enumerate}
\item[i)] the graph $\Gamma_m$ is $m$-isomorphic to  $\Gamma'_m$;
\item[ii)] the associated $R_m$-modules ${\displaystyle \mathbb{V}(\Gamma_m)}$ and ${\displaystyle \mathbb{V}(\Gamma'_m)}$ are isomorphic;
\item[iii)]  the representations $ \rho_{\Gamma_m}$ and $\rho_{\Gamma'_m}$  are equivalent.
\end{enumerate}
\end{thm}

\begin{proof}\mbox{}

\begin{itemize}
\item[-] $i)\Longrightarrow $ii) : let $(\varphi, \Id) : \Gamma_m \longrightarrow \Gamma'_m$ be an $m$-isomorphism: $\varphi \circ \alpha=\alpha'$; we have: 
$f(\alpha(e))=\alpha'(\Id(e))= \alpha'(e)$, consequently:
\begin{equation*}
\begin{array}{l}
\left((x\in \alpha(e)\Longleftrightarrow \varphi(x)\in \alpha'(e)\right)\Longleftrightarrow\\
\left(e\in \gamma(x)\Longleftrightarrow  e\in \gamma'(\varphi(x))\right)\Longleftrightarrow\\
 \left(\xi_{\gamma(x)}\oplus R_m\right)\simeq \left(\xi_{\gamma'\varphi(f(x))=x')}\oplus R_m\right).  
\end{array}
\end{equation*}
Since the decompositions of  both $\mathbb{V}(\Gamma)$ and $\mathbb{V}(\Gamma')$  in $\left(\xi_{\gamma(x)}\oplus R_m\right), x\in V$ and  
$\left(\xi_{\gamma'(x')}\oplus R_m\right),$ $\varphi(x)=x'\in V'$ is a bijection $(\varphi; \Id)$  can be extend to a $\overline{\mathbb{F}}_p$-isomorphism 
$\varphi : \mathbb{V}(\Gamma_m) \longrightarrow \mathbb{V}(\Gamma'_m)$.
  
 \noindent We can also see, considering that the above isomorphism is an isomorphism of abelian groups,   that for all $i\in\mathcal{E}_m$ we have:  ${\displaystyle \varphi(\alpha(i))=\alpha'(i)}$, so that
${\displaystyle  x \in \alpha(i) \Longleftrightarrow \varphi(x)\in \; \alpha'(i)}$ : 
${\displaystyle  \gamma(x)=\gamma'(\varphi(x))}$; so ${\displaystyle  \varphi(i.x)=\varphi(\chi_{\gamma(x)}(i)x)=\chi_{\gamma'(\varphi(x))}(i)\varphi(x)=i.\varphi(x)}$ : $\varphi$ is an isomorphism of $R$-modules;

\item[-] ii)$\Longrightarrow$ i) : suppose now that  $\mathbb{V}(\Gamma_m) \simeq \mathbb{V}(\Gamma'_m)$ as $R_m$-modules, they have the same canonical decomposition, of the form 
$$\mathbb{V}(\Gamma)\simeq \bigoplus_{a \in \mathbb{E}_m}m_a (\xi_a\oplus R_m)=
 m_{is} \left(\xi_{0}\oplus R_m\right)\bigoplus_{e\in I_{\mathcal{E}_m}} 2 \left(\xi_{e}\oplus R_m\right) \bigoplus_{x \in V_{\gamma}} \left(\xi_{\gamma(x)}\oplus R_m\right)$$
and
$$\mathbb{V}(\Gamma')\simeq \bigoplus_{a \in \mathbb{E}_m}m'_a (\xi_a\oplus R_m)=
 m'_{is} \left(\xi_{0}\oplus R_m\right)\bigoplus_{e'\in I'_{\mathcal{E}_m}} 2 \left(\xi_{e}\oplus R_m\right) \bigoplus_{x' \in V'_{\gamma'}} \left(\xi_{\gamma'(x')}\oplus R_m\right).$$
So we have $m_a=m'_a$ for all $a$.
\begin{itemize}
\item[-] Firstly the coefficients of $\xi_0 \oplus R_m$ are the same so $m_{is}=m'_{is}: \Gamma$ and $\Gamma'$ have the same number of isolated vertices;
\item[-] Secondly $\{a: m_a=2\}=\{ a: m'_a=2\}$ so that $ I_{\mathcal{E}_m}= I'_{\mathcal{E}_m}$ and  $\Gamma, \Gamma'$ have the same number of isolated edges;
\item[-] Thirdly $\{a: m_a=1\}=\{ a: m'_a=1\}$ and for all  $x \in V_{\gamma}$ there exists an  unique $x' \in
 V'_{\gamma'}$ such that $\xi_{\gamma(x)}=\xi_{\gamma'(x')}$, that is $\gamma(x)
 =\gamma'(x')$ which   gives a bijection: 
\end{itemize}
 \begin{eqnarray*} 
 \varphi  : V_{\gamma}& \longrightarrow V'_{\gamma'}\\
                               x    & \longmapsto x'
\end{eqnarray*}
The map $\varphi$ preserves the  incidence : if $i$ is  incident to $x$ in $\Gamma_m$, we have $i \in
\gamma(x)$; so, from above  we have: $i \in \gamma'(x')$: $i$ is incident to $x'$ in $\Gamma'_m$ : $\varphi\circ\alpha(i)=\alpha'(i)$.\\

\noindent Finally, we complete $\varphi$ by a bijection between isolated vertices, and a bijection between isolated edges.
\item[-] $ii)\Longleftrightarrow $iii) : from above. \qedhere
\end{itemize}
\end{proof}


\section{Character of the $\overline{\mathbb{F}}_p$-represensation $\rho_{\Gamma}$ of $\mathbb{E}$ }

\noindent  Let ${\displaystyle \rho_{\Gamma} : \mathbb{E} \longrightarrow GL( \mathbb{V})}$ be the representation associated to the graph $\Gamma=(V;E,\varepsilon)$. 
Classically  the character $ \chi_{\rho_{\Gamma}}: \mathbb{E} \longrightarrow \overline{\mathbb{F}}_p$ associated with the representation $\rho_{\Gamma}$ is 
\begin{equation}\label{characterrep}
\chi_{\rho_{\Gamma}}(a):= \Tr(\rho_{\Gamma}(a)),  \;\; a \in \mathbb{E}.
\end{equation}
Depending on the context, we will sometimes write $\chi_{\rho_{\Gamma}}$ as $\chi_{\rho}$ or $\chi_{\Gamma}$.\\
By expanding the formula \ref{characterrep}, we obtain:
\begin{equation}
\chi_{\rho}(a) = \sum_{x \in V}(-1)^{| \sigma(a)\cap \sigma(\gamma(x))|}= \sum_{x \in V}(-1)^{deg_{\Gamma(a)}(x)} \in \mathbb{F}_p,
\end{equation}
where $\Gamma(a)$ is the subgraph induced by the edges of $a$. By convention, we will write: $ (-1)^{deg_{\Gamma(0)}(x)}=(-1)^{deg_{\emptyset}(x)} =0$.\\
We define 

\begin{equation}
s_{\rho}(a)=  \sum_{    x \in V(\Gamma(a))   }   (-1)^{deg_{\Gamma(a)}(x)  } \in \mathbb{F}_p
\end{equation}
the  {\bf signature} of the character $\chi_{\rho}$, and  the {\bf thinness of  $\chi_{\rho}$} by
\begin{equation}
t(a)= \vert(V(\Gamma(a))^c\vert.
\end{equation}
We denote $\overline{t(a)}= t(a) \mod p$, and more generally for $s \in \mathbb{Z}: \overline{s}= s \mod p.$\\
Thus, we can write:
\begin{equation}
\chi_{\rho}(a) = s_{\rho}(a)+ \overline{t(a)}.
\end{equation}


\begin{thm}  Let  $\Gamma_m=(V;\mathcal{E}_m, \alpha)$ and $\Gamma'_m=(V'; \mathcal{E}_m, \alpha')$ be two connected $m$-graphs; suppose  that $p > \max\{ |V|, |V'| \}$;
the following assertions are equivalent:
\begin{itemize}
\item[i)]  $\Gamma_m \simeq \Gamma'_m$;
\item[ii)] $\chi_{\Gamma_{m}} = \chi_{\Gamma'_{m}}.$
\end{itemize}
\end{thm}

\begin{proof} \mbox{}
\begin{itemize}
\item[-] If $\chi_{\Gamma_m} = \chi_{\Gamma'_m}$ by applying Theorem  \ref{frobenius}  and  Theorem  \ref{graphEquivModules}  we have:
$$ \forall \; a \; \; m_a \equiv m'_a  \; \mod p$$
but $m_a, m'_a < |V|,|V'| < p$ so that $m_a=m'_a$ and $\mathbb{V} \simeq \mathbb{V}'$ so 
$\Gamma_m \simeq \Gamma'_m$.
\item[-]  Assume now that $\Gamma_m \simeq \Gamma'_m$; by Theorem  \ref{graphEquivModules}: $ {\displaystyle \rho_{\Gamma_m}\simeq \rho_{\Gamma'_m}}$,
it exists $\varphi: W \longrightarrow W'$ a $ \overline{\mathbb{F}}_p$-isomorphism such that for all $a \in {E}_m$
\[
\varphi\circ \rho (a)\circ  \varphi^{-1}= \rho'(a)
\]
$\chi_{\Gamma_{m}}= \Tr(\rho(a))= \Tr(\varphi\circ \rho (a)\circ  \varphi^{-1})= \Tr(\rho'(a))=\chi_{\Gamma'_{m}}$.\qedhere
\end{itemize}
\end{proof}

Some properties of  $\chi_{\rho}$ and $s_{\rho}$ are recorded in the following.


 \begin{prop}\textrm{}\\
 Let $\Gamma=(V;E,\varepsilon)$ be a simple connected graph with $|V|=n \ge 2$ and $|E|=m$; we have the following properties:
 
\begin{itemize}
\item[i)] $\chi_{\rho}(0) =  \overline{n} $;
\item[ii)]  ${\displaystyle \sum_{e\in E}\chi_{\gamma(x)}(e)= \overline{-deg(x)+ t(a)}}$ and  ${\displaystyle \sum_{e\in \gamma(x)}\chi_{\gamma(x)}(e)=\overline{-deg(x)}}$;
\item[iii)] if $\Gamma(a)=L_k$ the line with $k$ edges, then  $\chi_{\rho}(a)=\overline{|\sigma(a)|-4 +t(a)}$;
\item[iv)] if $\Gamma(a)=C_k$ the cycle with $k$ edges, then $\chi_{\rho}(a)=\overline{|\sigma(a)| + t(a)}$;
\item[v)]  if $\Gamma(a)=K_k$ the complete graph with $k$ vertices, then $s_{\rho}(a)=(-1)^{k-1}\overline{|\sigma(a)|} +
\overline{t(a)}$.
\end{itemize}
\end{prop}

\begin{proof}\textrm{}\\
 
 \begin{itemize}
\item[i)] $ \chi_{\rho}(0) =   \overline{n}  $ because  $(-1)^{deg_{\emptyset}(x)} =1$;
\item[ii)]we have: ${\displaystyle \chi_{\gamma(x)}e=(-1)^{deg_{\Gamma(e)}(x)}}$, if $x\in \varepsilon(e)$ we obtain $-1$ otherwise we obtain $1$, hence the result follows;
\item[iii)] ${\displaystyle  V(\Gamma(a))=\{x_1,\ldots, x_{k+1} \}}$, with ${\displaystyle  a=e_1 +\cdots +e_k}$ and 
$ \varepsilon(e_i)=\{x_i, x_{i+1} \} $:\\
$ s_{\rho}(a)=-1 +\underbrace{1+ \cdots +1}_{k-2} -1= k-4  \mod p = \overline{|\sigma(a)| -4}$;
\item[iv)] in this case:  $  x_{k+1}=x_1$, so $  s_{\rho}(a)=\underbrace{1+ \cdots +1}_{k}=\overline{k}=\overline{|\sigma(a)|}$;
\item[v)] it is   clear.\qedhere
\end{itemize}
\end{proof}

Let's set $a_m=\oplus_{e \in E}e$, then 


\begin{cor}\textrm{}\\
i) if $\Gamma= L_k$, then $\chi_{\rho}(a_m)= m-4  \mod p$;\\
ii) if $\Gamma= C_k$, then $\chi_{\rho} (a_m)=m \mod p $;\\
iii) if $\Gamma=K_k$, then $\chi_{\rho}(a_m)= (-1)^{k-1} \frac{k(k-1)}{2} \mod p$.
\end{cor}

\begin{proof} Clear since $t(a_m)=0$ (the graph is without isolated vertex).
\end{proof}
\noindent Other results are given in:

\begin{prop}\textrm{}
 Let $\Gamma=(V;E,\varepsilon)$ be a simple connected graph with $|V|=n \ge 2$ and $|E|=m$;  moreover choose $p >n+2m$, then
\begin{itemize}
\item[i)]  $s_{\rho}(a)=-2 |\sigma(a)| \mod p \Longleftrightarrow \Gamma(a)$ is a matching;
\item[ii)]  $\chi_{\rho}(a_m)=n  \mod p \Longleftrightarrow \Gamma$ is Eulerian;
\item[iii]  $\chi_{\rho}(a)=s_{\rho}(a)  \Longleftrightarrow \Gamma(a)$ is a cover of $\Gamma$ (that is 
the edges of $a$ are  incident to any vertex of  $\Gamma$).
\end{itemize}
\end{prop}

\begin{proof}\textrm{}
\begin{itemize}
\item[i)]  Assume that $\Gamma(a)$ is a matching: $\varepsilon(e) \cap \varepsilon(e')= \emptyset$ for all distinct $e,e' \in \sigma(a)$ so that $x \in V(\Gamma(a))  \Longrightarrow \deg_{\Gamma(a)}x=1$ and $s_{\rho}(a) =- |V(\Gamma(a))| \mod p$; but 
$| V(\Gamma(a))  |=  2 |\sigma(a)|$,  so $s_{\rho}(a)=-2 |\sigma(a)| \mod p$.\\\\
Reciprocally  suppose $s_{\rho}(a)=-2 |\sigma(a)| \mod p$; define \\
$S=| \{ x \in V(\Gamma(a)) : \deg(x)  \text{ is even} \}|$\\
$T=| \{ x \in V(\Gamma(a)) : \deg(x)   \text{ is odd} \}|.$\\
We have $s_{\rho}(a)= S-T \mod p$, that is 
\[
T \equiv S+ 2 |\sigma(a)| \; \mod p
\]
but $0 \le T <p$ and $0 \le S+2|\sigma(a)| \le n+2m <p$ this implies that $ T= S+ 2 |\sigma(a)|$; moreover, 
$|V(\Gamma(a))|=S+T \le 2 |\sigma(a)|$ hence $T=S+2|\sigma(a)| \ge S+T$: necessarily $S=0$ and we deduce 
$|V(\Gamma(a))| =T=2 |\sigma(a)|$: this implies that all the edges of $a$ are disjoint; hence, $\Gamma(a)$ is a  matching;
\item[ii)]  $\Gamma(a_m)= \Gamma$ since $\Gamma$ is connected: $t(a_m)=0$ and $\chi_{\rho}(a_m)=\sum_{x \in V} (-1)^{deg(x)}=S-T$; $S+T=n$: so $\chi_{\rho}(a_m)=n=S+T \mod p $ iff $T=0  \mod p$  iff $\Gamma$ is Eulerian;
\item[iii)]  $\chi_{\rho}(a)=s_{\rho}(a)$ iff $t(a)=0$, but:
\begin{itemize}
\item[-]  if $t(a)=0 \in \mathbb{F}_p$, then $p$ divide $|V(\Gamma(a))^c|$; or this number is $\le n <p $ so that $V(\Gamma(a))^c= \emptyset$ and $V(\Gamma(a)=V: \Gamma(a)$ is a cover,
\item[-]  if $V(\Gamma(a))=V$ clearly $t(a)=0$. 
\end{itemize}

\end{itemize}
 \end{proof}


\section{Illustrative examples}

\medskip
\begin{examp}
Let $\Gamma=(V;E,\varepsilon)$ be the simple graph defined by
\[
V=\{x_1,x_2\}\; \textrm{and}\; 
E=\{e_1\},
\]
$m=1, n=2$\\
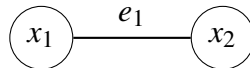
\begin{figure}[h]
\centering
\begin{tikzpicture}[scale=1.2]
    \node[circle, draw, fill=white, minimum size=0.8cm] (x1) at (0,0) {$x_1$};
    \node[circle, draw, fill=white, minimum size=0.8cm] (x2) at (2,0) {$x_2$};
    \draw[thick] (x1) -- (x2) node[midway, above] {$e_1$};
\end{tikzpicture}
\caption{\footnotesize Graph $\Gamma$ with $V=\{x_1,x_2\}$, $E=\{e_1\}$, and $\varepsilon(e_1)=\{x_1,x_2\}$.}
\label{fig:gamma}
\end{figure}

where the unique edge $e_1$ joins the two vertices, that is,
\[
\varepsilon(e_1)=\{x_1,x_2\}.
\]

Since $\mathbb{E}=\{0,e_1\}$ with $e_1\oplus e_1=0$, the group $\mathbb{E}$ is isomorphic to
$\mathbb{Z}/2\mathbb{Z}$.

\medskip
The vector space associated with the vertices is
\[
\mathbb{V}=\overline{\mathbb{F}}_{p}[V]= \overline{\mathbb{F}}_{p}x_1 \oplus \overline{\mathbb{F}}_{p} x_2.
\]

Now we define the map $\gamma : V \to \mathbb{E}$ : for each vertex $x\in V$:
\[
\gamma(x)=\underset{x \in\varepsilon(e)}{\oplus} e.
\]
Hence,
\[
\gamma(x_1)=e_1, \qquad \gamma(x_2)=e_1.
\]

  For $x\in V$ and $e\in E$, $\chi_{\gamma(x)}$ satisfies
\[
\chi_{\gamma(x)}(e)=
\begin{cases}
-1 & \text{if } x\in\varepsilon(e),\\
\;\;1 & \text{otherwise}.
\end{cases}
\]
Therefore, we get
\[
\chi_{\gamma(x_1)}(e_1)=-1, ~~ and ~~
\chi_{\gamma(x_2)}(e_1)=-1.
\]

\noindent The definition of the representation:  $\rho_\Gamma: \mathbb{E}\to\Aut_{\mathbb{F}_{p}}(\mathbb{V})$ is given by:
\[
\rho_\Gamma(a)(x)=\chi_{\gamma(x)}(a)\,x,\; \textrm{where}\; 
 a\in \mathbb{E},\; x\in V.
\]

\noindent For the neutral element $0\in  \mathbb{E}$, we have
\[
\rho_\Gamma(0)(x)=x,
\]
which implies that $\rho_\Gamma(0)=\mathrm{Id}_V$.

\noindent For $a=e_1$, we obtain:
\[
\rho_\Gamma(e_1)(x_1)=-x_1, 
\qquad
\rho_\Gamma(e_1)(x_2)=-x_2.
\]
\end{examp}


\begin{examp}
Consider the path graph
\[
V= \{x_1, x_2 , x_3\}
\qquad
E=\{e_1,e_2\}.
\]
$m=2,~ n=3$\\
that is

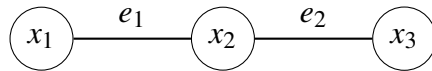
\begin{figure}[h]
\centering
\begin{tikzpicture}[scale=1.2]
    \node[circle, draw, fill=white, minimum size=0.8cm] (x1) at (0,0) {$x_1$};
    \node[circle, draw, fill=white, minimum size=0.8cm] (x2) at (2,0) {$x_2$};
    \node[circle, draw, fill=white, minimum size=0.8cm] (x3) at (4,0) {$x_3$};
    
    \draw[thick] (x1) -- (x2) node[midway, above] {$e_1$};
    \draw[thick] (x2) -- (x3) node[midway, above] {$e_2$};
\end{tikzpicture}
\caption{\footnotesize Path graph with $V=\{x_1,x_2,x_3\}$ and $E=\{e_1,e_2\}$, where $\varepsilon(e_1)=\{x_1,x_2\}$, $\varepsilon(e_2)=\{x_2,x_3\}$.}
\label{fig:path-graph}
\end{figure}

\begin{itemize}
\item For $a=e_1$:
\[
\deg_{\Gamma(e_1)}(x_1)=\deg_{\Gamma(e_1)}(x_2)=1,
\]

\[
s_{\rho}(e_1)=(-1)^1+(-1)^1=-2, t(e_1)=1
\]
\[
\chi_{\rho}(e_1)= -1
\]
\item For $a=e_2$:
\[
\deg_{\Gamma(e_2)}(x_2)=\deg_{\Gamma(e_2)}(x_3)=1,
\]
\[
s_{\rho}(e_2)=(-1)^1+(-1)^1=-2, t(e_2)=1
\]
thus
\[
\chi_{\rho}(e_2)= -1
\]

\item For $a=e_1+e_2$:
\[
\deg(x_1)=1,\quad \deg(x_2)=2,\quad \deg(x_3)=1,
\]
\[
s_{\rho}(e_1 \oplus e_2)=(-1)^1+(-1)^2+(-1)^1=-1, t(e_1 \oplus e_2)=0
\]
\[
\chi_{\rho}(e_1\oplus e_2)= -1
\]

\end{itemize}
\end{examp}

\subsection{Examples for m=2} \mbox{}\\

Here $E=\{ e_1,,e_2 \}$.
We use the following "order" for $\mathbb{E}; 0, e_1, e_2, e_1\oplus e_2$, so that 
$$(m_a)_{a \in \mathbb{E}}=(m_0, m_{e_1}, m_{e_2}, m_{e_1 \oplus e_2}) \in \mathbb{N}^4.$$

In the table,  $I_n$ denote $n$ isolated vertices; the values $\chi_{\Gamma}(e)$ are integers modulo p.\\\\
\begin{center}
\renewcommand{\arraystretch}{1.2}
\begin{tabular}{|l|c|c|c|c|c|}
\hline
Graph $\Gamma$   & n  & $(m_a)$ & $\chi_{\Gamma}(e_1)$  & $\chi_{\Gamma}(e_2)$  & $\chi_{\Gamma}(e_1 \oplus e_2)$  \\
\hline 
$L_2$                      & 3  & (0,1,1,1) &             -1                      &                      -1             &             -1                       \\
\hline
$L_2+I_1$               & 4  & (1,1,1,1) &             0                        &                        0           &             0                         \\
\hline
$2L_1$                    & 4  & (0,2,2,0) &               0                        &                        0          &             -4                          \\
\hline
$2L_1+I_1$             & 5  & (1,2,2,0) &               1                       &                      1               &            -3                         \\
\hline
\end{tabular}
\end{center}\mbox{}\\

\noindent As example the canonical decomposition for $\Gamma= L_2+I_1$ is 
$$ \mathbb{V}= (\xi_0 \oplus R) \oplus (\xi_{e_1} \oplus R) \oplus (\xi_{e_2} \oplus R) \oplus (\xi_{e_1\oplus e_2} \oplus R).$$

\begin{rem} If $m_0$ is the number of isolated vertices, $m_2$ the number of isolated edges, $m_1$ the number of remaining vertices  we have:
\[
n=m_0 + 2m_2+ m_1.
\]
\end{rem}

\subsection{Examples for m=3} \mbox{}\\
Here the "order" is $(m_a)=(m_0,m_{e_1}, m_{e_2},m_{e_3}, m_{e_1\oplus e_2},  m_{e_1\oplus  e_3}, m_{e_2 \oplus  e_3}, m_{e_1\oplus e_2
\oplus e_3})$.\\
Example of calculation of the character of $\Gamma=L_3$:\\
$\gamma(x_1)=e_1;\\
\gamma(x_2)=e_1\oplus e_2;\\
\gamma(x_3)=e_2 \oplus e_3;\\
\gamma(x_4)= e_3;$\\
with the canonical decomposition: 
$$ \mathbb{V}= (\xi_{e_1} \oplus R) \oplus (\xi_{e_3} \oplus R) \oplus (\xi_{e_1\oplus e_2} \oplus R) \oplus  (\xi_{e_2\oplus e_3} \oplus R).$$.

\begin{center}
\renewcommand{\arraystretch}{1.2}
\begin{tabular}{|c|c|c|c|c|c|c|c|c|}
\hline
$\mathbb{E}$       & 0  & $e_1$ & $e_2$ & $e_3$ & $e_1\oplus e_2$  &  $e_1\oplus e_3$ &  $e_2\oplus e_3$ &  $e_1\oplus e_2 \oplus e_3$ \\
\hline
$\chi_{\gamma(x_1)} $ &1  & -1     & 1          &   1     &       -1                    &           -1               &             1              &              -1 \\
\hline
$\chi_{\gamma(x_2)} $ &1  & -1     & -1          &   1     &       1                    &           -1               &            - 1              &              1 \\
\hline
$\chi_{\gamma(x_3)} $ &1  & 1     &  -1          &  - 1     &       -1                    &           -1               &             1              &              1 \\
\hline
$\chi_{\gamma(x_4)} $ &1  & 1     & 1          &   -1     &       1                    &           -1               &            - 1              &              -1 \\
\hline
$\chi_{\Gamma} $ &4  & 0     & 0         &   0    &       0                    &           -4               &             0              &              0 \\
\hline
\end{tabular}
\end{center}\mbox{}\\\\
In this table  the rows  are "multiplicative", and the columns are " additive" (modulo p). \\\\
Now some examples with $m=3$. The last column is 
$$ \chi_{\Gamma}\;:\; [\chi_{\Gamma}(e_1), \chi_{\Gamma}(e_2), \chi_{\Gamma}(e_3); \chi_{\Gamma}(e_1\oplus e_2), 
\chi_{\Gamma}(e_1\oplus e_3), \chi_{\Gamma}( e_2 \oplus e_3); \chi_{\Gamma}(e_1\oplus e_2 \oplus e_3)].
$$ \\
\begin{center}
\renewcommand{\arraystretch}{1.4}
\begin{tabular}{|c|c|c|c|}
\hline
Graphs      & $n$    & $(m_a)_{a \in \mathbb{E}}$   &  $\chi_{\Gamma}$ \\ 
\hline
$C_3$       & 3    & (0,0,0,0,1,1,1,0)                      & [-1,-1,-1;-1,-1,-1;3]\\
\hline
$L_3$        &  4  & (0,1,0,1,1,0,1,0)                      & [0,0,0;0,-4,0;0]      \\
\hline
$K_{1,3}$   & 4  & (0,1,1,1,0,0,0,1)                      & [0,0,0; 0,0,0;-4]       \\
\hline
$L_2+L_1$  & 5  & (0,1,2,1,1,0,0,0)                     & [1, 1,1,  1,-3,-3;-3]      \\
\hline
\end{tabular}
\end{center}\mbox{}\\\\


\section{Conclusion}

This article is an introduction to the theory of linear group representations applied to graph theory. We associate a representation to a graph, which allows us to decompose the graph canonically into modules.\\
The character of the representation allows us to highlight combinatorial properties of the graph. In our opinion, it would be interesting to explore this relationship further and, moreover, to demonstrate a correspondence between the combinatorial properties of graphs and the algebraic properties of the modules $\mathbb{V}(\Gamma)$. This could help to clarify known conjectures such as the \textsc{Kelly-Ullam} conjecture (reconstruction conjecture). The algorithmic aspect is also a promising avenue to explore and the development of algorithms to calculate the linear representation associated with a graph as well as its associated character will be the subject of a future article.\\
We have worked with the  \textsc{Galois} field $\mathbb{F}_{2}$, but in our work there should be no obstacle to using another field, $\mathbb{F}_{q}$, where $q$ is the odd prime, with $q\neq p$. For this, it would suffice to redefine the linear characters on the basis, which, in our opinion, is not very difficult.\\
Furthermore, other graph properties (coloring, stability, kernel, ...) should be visible thanks to the characters; it would be interesting to highlight them.\\
Finally, a generalization to hypergraphs, \cite{bre4} seems relatively straightforward. By modifying the incidence relation $\rho_{\Gamma}$, by taking $q\neq2$ to be prime, or by giving a different definition of the function $\gamma$, we could obtain more general results.\\
Other approaches involve studying directed (hyper)graphs and weighted (hyper)graphs.

\end{document}